\documentclass[a4paper,11pt]{amsart}
\usepackage[T1]{fontenc}
\usepackage[utf8]{inputenc} 
\usepackage{lmodern}
\usepackage{amsmath,amssymb,amsthm}
\usepackage[english]{babel}
\usepackage{graphicx}
\usepackage{tikz}

\newcommand{\regine}[1]{}
\newcommand{\olivier}[1]{}
\newcommand{\jb}[1]{}

\newcommand{\N}{\mathbb N}
\newcommand{\Z}{\mathbb{Z}}
\newcommand{\Zd}{\mathbb{Z}^d}

\newcommand{\R}{\mathbb{R}}
\newcommand{\Rd}{\mathbb{R}^d}

\newcommand{\Edo}{\overrightarrow{\mathbb{E}}^d}
\newcommand{\Eddo}{\overrightarrow{\mathbb{E}}^{d+1}_{\text{alt}}}

\renewcommand{\P}{\mathbb{P}}
\newcommand{\E}{\mathbb{E}}
\newcommand{\Ebarre}{\overline{\mathbb{E}}}
\newcommand{\Pbarre}{\overline{\mathbb{P}}}

\renewcommand{\epsilon}{\varepsilon}
\renewcommand{\phi}{\varphi}
\renewcommand{\limsup}{\overline{\lim}}
\renewcommand{\liminf}{\underline{\lim}}

\newcommand{\ie}{\emph{i.e. }}

\newcommand{\miniop}[3]{%
\renewcommand{\arraystretch}{0.6}
\begin{array}{c}
{\scriptstyle #1}\\
#2\\
{\scriptstyle #3}
\end{array}
\renewcommand{\arraystretch}{1}}

\newcommand{\pcdirdeux}{\overrightarrow{p_{c,2}}^{\text{alt}}}
\newcommand{\pcdirtrois}{\overrightarrow{p_{c,3}}^{\text{alt}}}
\newcommand{\pcdir}{\overrightarrow{p_c}^{\text{alt}}}

\newtheorem{theorem}{Theorem}[section]
\newtheorem{lemm}[theorem]{Lemma}

\newtheorem{rema}[theorem]{Remark}
\newtheorem{prop}[theorem]{Proposition}

\title{The Number of Open Paths in Oriented Percolation}
\author{Olivier Garet, Jean-Baptiste Gou\'er\'e and R\'egine Marchand}

\begin{document}

\begin{abstract}
We study the number $N_n$  of open paths of length $n$ in supercritical oriented percolation on $\Zd \times \N$, with $d \ge 1$. We prove that on the percolation event $\{\inf N_n>0\}$, $N_n^{1/n}$ almost surely converges to a positive deterministic constant. We also study the existence of directional limits.
The proof relies on the introduction of adapted sequences of regenerating times, on subadditive arguments and on the properties of the coupled zone in supercritical oriented percolation.
\end{abstract}

\maketitle

\section{Introduction and main results}

\subsection*{Introduction} Consider supercritical oriented percolation on $\Z^d \times \N$.
Let $N(a,b)$ denote the number of open paths from $a$ to $b$.
By concatenation of paths we get $N(a,c) \ge N(a,b)N(b,c)$.
In other words, the following superadditivity property holds:
$$
\log N(a,c)\ge \log N(a,b)+\log N(b,c).
$$
Having in mind subadditive ergodic theorems, it seems then natural to think that, on the percolation event "the cluster of the origin is infinite", the number $N_n$ of open paths with length $n$  starting from the origin should grow exponentially fast in $n$. 
However, the possibility for edges to be closed implies that $\log N(\cdot,\cdot)$ may be infinite, and therefore not integrable. 
This prevents from using subadditive techniques, at least in their simplest form. 
The growth rate of $N_n^{1/n}$ and related objects have already been studied.
Fukushima and Yoshida \cite{MR3069368} proved that  $\liminf N_n^{1/n}$ is almost surely strictly positive on the percolation event.	
Lacoin~\cite{MR2928724} proved that the straigtforward inequality $\limsup N_n^{1/n} \le p (2d+1)$ is not always an equality.
In spite of those works, to our knowledge, there was no proof of the convergence of $N_n^{1/n}$ in the literature.

Such a convergence has been obtained for a relaxed kind of percolation called $\rho$-percolation. 
Let $\rho \in (0,1)$ and let $N_n(\rho)$ denotes the number of paths with length $n$ using at least $\rho n$ open edges.
The existence of the limit $N_n(\rho)^{1/n}$  has been proved in Comets--Popov--Vachkovskaia~\cite{MR2386584} 
and in Kesten--Sidoravicius~\cite{MR2643568} by different methods.

The present paper aims to prove that in supercritical oriented percolation,    
$N_n^{1/n}$ has an almost sure  limit on the percolation event.
The proof relies on essential hitting times which have been introduced in Garet--Marchand~\cite{GM-contact} 
in order to establish a shape theorem for the contact process in random environment. 
Let us now define precisely the oriented percolation setting we work with.

%

\subsection*{Oriented percolation in dimension $d+1$.}
Let $d \ge 1$ be fixed,  and let $\|.\|_1$ be the norm on $\mathbb R^d$ defined by
$$\forall x =(x_i)_{1 \le i \le d} \in \Rd \quad \|x\|_1=\sum_{i=1}^d|x_i|.$$
We consider the oriented graph whose set of sites is  $\Zd \times \N$,
where $\N=\{0,1,2,\dots\}$, and we put an oriented edge from $(z_1,n_1)$ to $(z_2,n_2)$ if and only if $$n_2=n_1+1\text{ and } \|z_2-z_1\|_1\le 1;$$ the set of these edges is denoted by $\Eddo$. We say that $\gamma=(\gamma_i,i)_{m \le i \le n} \in (\Zd \times \N)^{n-m+1}$ is a \emph{path} if and only if
$$\forall i \in \{m,\dots,n-1\} \quad \|\gamma_{i+1}-\gamma_i\|_1 \le 1.$$
Fix now a parameter $p \in [0,1]$, and open independently each edge with probability $p$. More formally, consider the probability space $\Omega=\{0,1\}^{\Eddo}$, endowed with its Borel $\sigma$-algebra and the probability $$\P_p=(\textrm{Ber}(p))^{\otimes \Eddo},$$ where $\textrm{Ber}(p)$ stands for the Bernoulli law of parameter $p$. For a configuration $\omega=(\omega_e)_{e \in \Eddo} \in \Omega$, say that the edge $e \in \Eddo$ is open 
if $\omega_e=1$ and closed otherwise. 
A path is said \emph{open} in the configuration $\omega$ if all its edges are open in $\omega$. For two sites $(v,m), (w,n)$ in $\Zd\times\N$, we denote by $\{(v,m) \to (w,n)\}$ the existence of an open  path from $(v,m)$ to $(w,n)$. 
By extension, we denote by $\{(v,m)\to +\infty\}$ 
the event that there exists an infinite open path starting from $(v,m)$.

There exists a critical probability $\pcdir(d+1)\in(0,1)$ such that: \jb{ref ?}\regine{plutot non}\jb{tu veux dire qu'il n'existe pas de reference ?}\regine{voir peut-etre Durrette 74-89, je ne l'ai pas sous la main}
$$\P_p((0,0) \to +\infty)>0 \; \Longleftrightarrow \; p > \pcdir(d+1).$$

In the following, we assume $p>\pcdir(d+1)$, and we will mainly work under the following conditional probability:
$$\Pbarre_p(.)=\P_p(.|(0,0) \to +\infty).$$

\subsection*{Global convergence result and previous results}
Denote by $N_n$ the number of open paths of length $n$ emanating from $(0,0)$.
Our main result is the following.

\begin{theorem}
\label{THEO-lebotheoun}
Let $p>\pcdir(d+1)$. There exists a strictly positive constant $\tilde \alpha_p(0)$ such that, $\Pbarre_p$-almost surely and in $L^1(\Pbarre_p)$,
$$
\miniop{}{\lim}{n\to +\infty} \frac{1}n\log N_n =\tilde \alpha_p(0).
$$
\end{theorem}

We now recall some questions related to this convergence problem.
First, note that $\mathbb E_p(N_n)=((2d+1)p)^n$.
As noticed by Darling~\cite{MR1178959}, the sequence $\left(N_n((2d+1)p)^{-n}\right)_{n\ge 0}$ is a  non-negative martingale, so there exists a non-negative random variable $W$ such that 
$$\P_p - a.s. \quad \frac{N_n}{(2d+1)^n p^n}\longrightarrow W\text{ and }\E_p[W]\le 1.$$
Therefore, it is easy to see that
$$\frac1{n}{\log N_n}\to \log((2d+1)p) \text{ on the event }\{W>0\}.$$
So when $W>0$, $N_n$ has the same growth rate as its expectation.
In his paper~\cite{MR1178959}, Darling was seeking for conditions implying that $W>0$. It seems that these questions have been forgotten for a while, but there is currently an increasing activity due to the links with random polymers -- see for example Lacoin~\cite{MR2928724} and Yoshida~\cite{MR2462010}.
Actually, it is not always the case that $W>0$. Let us summarize some known results:
\begin{itemize}
\item[$\bullet$] $\Pbarre_p(W>0)\in\{0,1\}$. The random variable 
$$
\chi=\miniop{}{\limsup}{n\to +\infty}\frac1{n}{\log N_n}$$
is $\Pbarre_p$-almost surely constant (see Lacoin~\cite{MR2928724}).
Note that a simple Borel-Cantelli argument ensures that $\chi\le \log ((2d+1)p)$.
\item[$\bullet$]  $W=0$ a.s. if $d=1$ or $d=2$ (see Yoshida~\cite{MR2462010}).  
\item[$\bullet$]  There exists $\pcdirdeux(d+1),\pcdirtrois(d+1)\in [\pcdir(d+1),1]$ such that:
\begin{itemize}
\item $\Pbarre_p(W>0)=1$ when $p>\pcdirtrois(d+1)$ and $\Pbarre_p(W>0)=0$ when $p<\pcdirtrois(d+1)$.
\item $\chi=\log(p(2d+1))$ $\Pbarre_p$-almost surely when $p>\pcdirdeux(d+1)$ and $\chi<\log(p(2d+1))$ $\Pbarre_p$-almost surely when $p<\pcdirdeux(d+1)$.
\item $\pcdirdeux(d+1) \le \pcdirtrois(d+1)$.
\item $\pcdirtrois(d+1) < 1$ if $d \ge 3$.
\end{itemize}
See Lacoin~\cite{MR2928724} Sections 2.2 and 2.3. 
\item It is believed that $\pcdirdeux(d+1)>\pcdir(d+1)$ and thus $\pcdirtrois(d+1)>\pcdir(d+1)$ when $d \ge 2$.
Lacoin~\cite{MR2928724} proved that the inequality is indeed strict for $L$-spread-out percolation for $d\ge 5$ and $L$ large.
\end{itemize}
In any case, it is clear that we need a proof of the existence of a limit for $\frac1{n}{\log N_n}$ that would not require  $W>0$. Our next result focuses on open paths with a prescribed slope.

\subsection*{Directional convergence results}

We first need to give a few more notations and results. Oriented percolation is known as the analogue in discrete time for the contact process. 
Usually, results are proved for one model, and it is commonly admitted that the proofs could easily be adapted to the other one. For the results concerning supercritical oriented percolation we use in this work, we will thus sometimes give the reference for the property concerning the contact process without any further explanation.

We define 
$$
\xi_n  = \{y \in \Zd:  \; (0,0)\to(y,n)\} \quad \text{ and } \quad 
H_n  = \miniop{}{\cup}{0\le k\le n}\xi_k.
$$
As for the contact process, the  growth of the sets $(H_n)_{n\ge 0}$  is governed by a shape theorem  when conditioned to survive: for every $p>\pcdir(d+1)$, there exists a norm $\mu_p$ on $\R^d$ such that for every $\varepsilon>0$, $\Pbarre_p$ almost surely, 
\begin{equation}
\label{E-shapetheo}
\exists N \; \forall n\ge N \;  B_{\mu_p}(0,(1-\varepsilon)n) \subset \;   H_n +[0,1]^d \; \subset B_{\mu_p}(0,(1+\varepsilon)n),
\end{equation}  
where $B_{\mu_p}(x,r)=\{y\in\Rd:\; \mu_p(y-x)\le r\}$. \regine{attention pour la derniere partie: par defaut, les boules sont fermees}
See, for the supercritical contact process, Durrett~\cite{MR1117232} or Garet-Marchand \cite{GM-contact}.

For every set $A \subset B_{\mu_p}(0,1)$, we denote by $N_{nA,n}$ the number of open paths starting from $(0,0)$, with length $n$ and whose extremity lies in $nA \cap \Zd$.

\begin{theorem}
\label{THEO-nombreAn}
Fix $p>\pcdir(d+1)$. There exists a concave function
$$\tilde \alpha_p: \mathring B_{\mu_p}(0,1) \longrightarrow (0, \log(p(2d+1))],$$
with the same symmetries as the grid $\Zd$, such that, for every set $A$   such that $\overline{\mathring{A}}=\overline{A} \subset \mathring B_{\mu_p}(0,1)$, $\Pbarre_p$-almost surely, 
$$  \miniop{}{\lim}{n\to +\infty} \frac{1}n\log N_{nA,n} 
= \sup_{x \in A} \tilde{\alpha}_p(x).$$
\end{theorem}
Since $\tilde\alpha_p$ is even and concave, the constant $\tilde\alpha_p(0)$ which appears in the statement of Theorem \ref{THEO-lebotheoun} is indeed the value
of the function $\tilde\alpha_p$ at $0$.

By considering, in Theorem \ref{THEO-nombreAn}, the  set  $A=B_{\mu_p}(x, \epsilon)$ for $x \in \mathring B_{\mu_p}(0,1)$ and for a small $\epsilon$, we see
that $\tilde{\alpha}_p(x)$ characterises the growth of the number of open paths with length $n$ and prescribed slope $x$. Using the very same technics of proof, one could for instance prove the following directional convergence result. If $x \in \Zd$, denote by $N_{x,n}$ the number of open paths from $(0,0)$ to $(x,n)$:

\begin{theorem} 
\label{CORO-limdir1}
Fix $p>\pcdir(d+1)$ and  $(y,h) \in \Zd \times \N^*$ such that $\mu_p(y)<h$. \\
\jb{Si on ne le prouve pas je serais plutot pour le mettre en simple remarque}
Extract from the sequence $(ny,nh)$ the (random) subsequence, denoted $\psi: \N \to \N$, of indices $k$ such that $(0,0) \to k.(y,h)$ . Then $\Pbarre_p$ almost surely, 
$$\lim_{n \to +\infty} \frac{1}{\psi(n)h} \log N_{\psi(n).(y,h)}=\tilde \alpha_p(y/h) .$$
\end{theorem}

Take now as a random environment a realization of oriented percolation on $\Zd \times \N$ with parameter $p$ such that $0$ percolates. Once this random setting is fixed, choose a random open path with length $n$, uniformly among all open paths with length $n$, and ask for the
behavior of the extremity of this random path.
More precisely, for every set $A$ with $\overline{\mathring{A}}=\overline{A} \subset \mathring B_{\mu_p}(0,1)$, the probability that the extremity of the random path stands in $nA$ is 
$$\frac{N_{nA,n}}{N_n}.$$
 Then, Theorem \ref{THEO-nombreAn} can be rephrased as 
a quenched large deviations principle for the extremity of this random open path:

\begin{rema}
\label{THEO-PGD}
Fix $p>\pcdir(d+1)$. For every set $A$ such that  $\overline{\mathring{A}}=\overline{A} \subset \mathring B_{\mu_p}(0,1)$,  $\Pbarre_p$-almost surely, 
$$  \miniop{}{\lim}{n\to +\infty} \frac{1}n\log \frac{N_{nA,n}}{N_n} 
= -\inf_{x \in A} \left( \tilde{\alpha}_p(0) - \tilde{\alpha}_p(x) \right).$$
\end{rema}
\subsection*{Open questions.}
Here are a few open questions.
\begin{itemize}
\item Is the following statement true ?
$$\forall x \in \mathring B_{\mu_p}(0,1) \backslash \{0_{\Zd}\} \quad \tilde\alpha_p(x) < \tilde\alpha_p(0).$$
If the statement held, then the extremity of a random open path with length $n$, uniformly chosen among open paths with length $n$, 
would concentrate near $0_{\Zd}$.
\item Is $\tilde\alpha_p$ strictly concave ? This would imply the previous statement.
\item Does $\tilde\alpha_p$ vanish when $x$ tends to the boundary of $\mathring B_{\mu_p}(0,1)$ ?
\end{itemize}

\subsection*{Organization of the paper.}
The key ideas of our proofs are the following. 

First, in Section 2, we recall results for supercritical oriented percolation, and we build the essential hitting times. 

Then, in Section 3, we fix a vector $(y,h) \in \Zd \times \N^*$ and we build an associated sequence of regenerating times $(S_n(y,h))_n$ (see Definition \ref{DEFI-Sn}). These random times satisfy $(0,0) \to (ny,S_n(y,h))\to +\infty$ and have good invariance and integrability properties with respect to $\Pbarre_p$. We can thus apply Kingman's subbaditive ergodic theorem to  obtain, in Lemma \ref{LEMM-SsAdd}, the existence of the following limit:
$$\frac{1}{S_n(y,h)}\log(N_{ny, S_n(y,h)}) \to \alpha_p(y,h).$$

Section 4 is devoted to the proof of Theorem \ref{THEO-lebotheoun}.
The asymptotic behavior of $\log (N_n)/n$ should come from the "direction" $(y,h)$ in which open paths are more abundant, \ie in the "direction" $(y,h)$ that maximizes $\alpha_p(y,h)$.
The key step to recover a full limit from the limit of a random subsequence is the continuity lemma \ref{LEMM-limsup}: using the coupled zone, we prove in essence that two points close in $\Zd\times \N^*$ and reached from $(0,0)$ by open paths should have similar number of open paths arriving to them. 

Finally, in Section 5,  the same ideas are used to prove Theorem \ref{THEO-nombreAn}.
The arguments are however more intricate. 
That is why we chose to present an independent proof of Theorem \ref{THEO-lebotheoun} where to our opinion, each type of argument -- regenerating time, coupling -- appears in a simpler form.

\subsection*{Notation.} 
For $n \ge 1$,  $x \in \mathbb Z^d$ and any set $A\subset \Rd$, we denote by 
\begin{itemize}
\item $N_n$ the number of open paths from $(0,0)$ to $\mathbb Z^d \times \{n\}$, 
\item $\overline{N}_n$ the number of open paths from $(0,0)$ to $\mathbb Z^d \times \{n\}$ that are the beginning of an infinite open path,
\item $N_{x,n}$ the number of open paths from $(0,0)$ to $(x,n)$, 
 \item $N_{A,n}$ the number of open paths from $(0,0)$ to $(A \cap \Zd) \times\{n\}$. 
\end{itemize}   
 
\section{Preliminary results}
  
\subsection{Exponential estimates for supercritical oriented percolation}
We work with the oriented percolation model in dimension $d+1$, as defined in the introduction. 
We set, for $n \in \N$ and $x \in \Zd$,
\begin{align*}
\xi^x_n & =  \{y \in \Zd:  \; (x,0)\to(y,n)\}, & & H^x_n =  \miniop{}{\cup}{0\le k\le n}\xi^x_k,\\
\xi_n^{\Zd} & =  \miniop{}{\cup}{x \in \Zd}\xi^x_n, & & K'^x_n  =   \miniop{}{\cap}{k\ge n}(\xi^x_k \Delta \xi_k^{\Zd})^c,\\
\tau^x & =  \min\{n \in \N:\; \xi^x_n =\varnothing\}.
\end{align*}
To simplify, we often write $\xi_n,\tau,H_n,K'_n$ instead of $\xi^0_n,\tau^0,H^0_n,K'^0_n$.

For instance, $\tau$ is the length of the longest open path starting from the origin, and the percolation event is equal to $\{\tau=+\infty\}$. First, finite open paths cannot be too long (see Durrett~\cite{MR1117232}):
\begin{equation}
\forall p>\pcdir(d+1) \quad  \exists A,B>0 \quad \forall n \in \N \quad 
\P_p(n\le \tau<+\infty)\le Ae^{-Bn}. \label{EQ-tauexp}
\end{equation}

The set $K'_n \cap H_n$ is called the coupled zone, and will play a central role in our proofs, by allowing to compare numbers of open paths with close extremities.
As for the contact process, the  growth of the sets $(H_n)_{n\ge 0}$ and the coupled zones $(K'_n\cap H_n )_{n\ge 0}$ is governed by a shape theorem and related large deviations inequalities:

\begin{prop}[Large deviations inequalities, Garet-Marchand \cite{GM-contact-gd}]
\label{PROP-GD}
$\;$ \\ Fix $p>\pcdir(d+1)$. For every $\epsilon>0$, there exist $A,B>0$ such that, 
$$
\forall n \ge 1 \quad \overline{\P}_p \left(\begin{array}{c}B_{\mu_p}(0,(1-\varepsilon)n) \subset \; (K'_n\cap H_n) +[0,1]^d\;\\ \subset \;H_n +[0,1]^d \; \subset B_{\mu_p}(0,(1+\varepsilon)n)\end{array} \right) \ge 1- Ae^{-Bn}.$$
\end{prop}

\begin{figure}[h!]
\begin{tikzpicture}
\draw (0,-0.05) node {$\bullet$} ;
\draw (0,0) node[below]{$(0,0)$} ;
\draw (0,0) -- (2.2,3);
\draw (0,0) -- (1.8,3);
\draw (0,0) -- (-2.2,3);
\draw (0,0) -- (-1.8,3);
\draw (-2.4,2.5) -- (2.4,2.5);
\draw (-2.4,0) -- (2.4,0);
\draw (1.1,2.45) node {$\bullet$} ;
\draw (1.1,2.5) node[above]{$(x,n)$};
\draw (-1.1,-0.05) node {$\bullet$} ;
\draw (-1.1,0) node[below]{$(y,0)$};
\draw[color=blue] (-1.1,0)--(-0.8,0.5)--(-0.1,1)--(0.5,1.5)--(0.2,2)--(1.1,2.5);
\draw[color=red] (0,0)--(-0.2,0.5)--(0.3,1)--(-0.2,1.5)--(0.6,2)--(1.1,2.5);
\end{tikzpicture}
\caption{Coupled zone.} If $x$ is in the coupled zone $K'^0_n$, and is reached by an open path starting from some point $(y,0) \in \Zd\times \{0\}$ (in blue), then $(0,0) \to (x,n)$ (in red).
\end{figure}
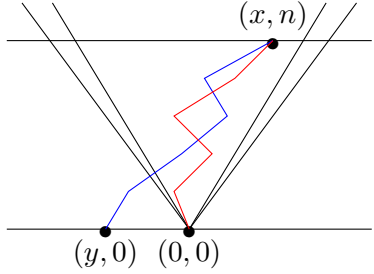

\subsection{Essential hitting times and associated translations}

We now introduce the analogues, in the discrete setting of oriented percolation, of the essential hitting times used by Garet--Marchand to study the supercritical contact process conditioned to survive in \cite{GM-contact} and  \cite{GM-contact-gd}; we give their main properties in Proposition \ref{PROP-sigma}. 

For a given $x \in \Zd$, the essential hitting time will be a random time $\sigma(x)$ such that
\begin{itemize}
\item $\Pbarre_p$ almost surely, $(0,0) \to (x, \sigma(x)) \to \infty$, 
\item the associated random translation of vector $(x, \sigma (x))$ leaves $\Pbarre_p$ invariant.
\end{itemize}
Thus $\sigma(x)$ will be interpreted as a regenerating time of the oriented percolation conditioned to percolate.

We define a set of oriented edges $\Edo$ of $\Zd$ in the following way: in $(\Zd,\Edo)$, there is an oriented edge between two points $z_1$ and $z_2$ in $\Zd$ if and only if  $\|z_1-z_2\|_1\le 1$.
The oriented edge in $\Eddo$  from $(z_1,n_1)$ to $(z_2,n_2)$ can be identified with the couple $((z_1,z_2),n_2)\in\Edo\times\N^*$. Thus, we identify $\Eddo$ and $\Edo\times\N^*$.
We also define, for $(y,h) \in \Zd \times \N$,  the translation $\theta_{(y,h)}$ on $\Omega$  by:
$$\theta_{(y,h)}((\omega_{(e,k)})_{e\in\Edo,k\ge 1})=(\omega_{(e+y,k+h)})_{e\in\Edo,k\ge 1}.$$
At some point, we will also need to look backwards in time. So, as set of sites, we replace $\Zd\times\N$ by  $\Zd\times\Z$, and we introduce the following reversed time translation defined on $\{0,1\}^{\Zd\times\Z}$ by
$$\theta^{\downarrow}_{(y,h)}((\omega_{(e,k)})_{e\in\Edo,k\in\Z})=(\omega_{(e+y,h-k)})_{e\in\Edo,k\in\Z}.$$
Fix $p > \pcdir(d+1)$. 

We now recall the construction of the essential hitting times   and the associated translations introduced in \cite{GM-contact}. Fix $x \in \Zd$. The essential hitting time $\sigma(x)$
 is defined through a family of stopping times as follows: we set 
 $u_0=v_0=0$
and we define recursively two increasing sequences of stopping times $(u_n)_{n \ge 0}$ and $(v_n)_{n \ge 0}$ with
$u_0=v_0< u_1< v_1< u_2\dots$ as follows:
\begin{itemize}
\item[$\bullet$] Assume that $v_k$ is defined. We set $u_{k+1}  =\inf\{t> v_k: \; x \in \xi^0_t \}$. \\
If $v_k<+\infty$, then $u_{k+1}$ is the first time after $v_k$ where $x$ is once again infected; otherwise, $u_{k+1}=+\infty$.
\item[$\bullet$] Assume that $u_k$ is defined, with $k \ge 1$. We set $v_k=u_k+\tau^0\circ \theta_{(x,u_k)}$.\\
If $u_k<+\infty$, the time $\tau^0\circ \theta_{(x,u_k)}$ is the length of the oriented percolation cluster starting from  $(x,u_k)$; otherwise, $v_k=+\infty$.
\end{itemize}
We then set
\begin{equation*}
K(x)=\min\{n\ge 0: \; v_{n}=+\infty \text{ or } u_{n+1}=+\infty\}.
\end{equation*}
This quantity represents the number of steps before the success of this process: either we stop because we have just found an infinite $v_n$, which corresponds to a time $u_n$ when $x$ is occupied and has infinite progeny, or we stop because we have just found an infinite $u_{n+1}$, which says that after $v_n$, site $0$ is never infected anymore. It is not difficult to see that 
$$\P_p(K(x)>n)\le\P_p(\tau^0<+\infty)^n,$$ and thus  $K(x)$ is $\P_p$ almost surely finite.  We define the essential hitting time $\sigma(x)$  by setting 
$$\sigma(x)=u_{K(x)} \in \N \cup \{+\infty\}.$$
By construction $(0,0)\to (x, \sigma(x)) \to +\infty$ on the event $\{\tau=+\infty\}$.
Note however that $\sigma(x)$ is not necessarily the first positive time when $x$ is occupied and has infinite progeny: for instance, such an event can occur between $u_1$ and $v_1$, being ignored by the recursive construction. It can be checked that  conditionally to the event $\{\tau^0=\infty\}$, the process necessarily stops because of an infinite $v_n$, and thus $\sigma(x)<+\infty$. 
At the same time, we define the operator $\tilde \theta$ on $\Omega$, which is a random translation, by:
\begin{equation*}
\tilde \theta_x(\omega) = 
\begin{cases} \theta_{(x,\sigma(x))} \omega & \text{if $\sigma(x)<+\infty$,}
\\
\omega &\text{otherwise.}
\end{cases}
\end{equation*}
If $(x_1,\dots,x_m)$ is a sequence of points in $\Zd$, we also introduce the shortened notation 
$$\tilde{\theta}_{x_1,\dots,x_m}=\tilde{\theta}_{x_m}\circ\tilde{\theta}_{x_m-1}\dots \circ \tilde{\theta}_{x_1}.$$

For $n\ge 1$, we denote by $\mathcal{F}_n$ the $\sigma$-field generated by the maps $(\omega\mapsto \omega_{(e,k)})_{e\in\Edo, 1\le k\le n}$. We denote by $\mathcal{F}$ the $\sigma$-field generated by the maps $(\omega\mapsto \omega_{(e,k)})_{e\in\Edo, k\ge 1}$.

\begin{prop}
\label{PROP-sigma}
Fix $p>\pcdir(d+1)$ and $x_1,\dots,x_m \in \Zd$.
\begin{enumerate}
\item[a.] Suppose $A\in\mathcal{B}(\R)$, $B\in\mathcal{F}$. Then for each $x\in\Zd$,
$$\Pbarre_p(\sigma(x)\in A, \tilde{\theta}_{x}^{-1}(B))=\Pbarre_p(\sigma(x)\in A)\Pbarre_p(B).$$
\item[b.] The probability measure $\Pbarre_p$ is invariant under  $\tilde{\theta}_{x_1,\dots,x_m}$.
\item[c.]  The random variables $\sigma(x_1),\sigma(x_2)\circ \tilde{\theta}_{x_1},\sigma(x_3)\circ \tilde{\theta}_{x_1,x_2},\dots,\sigma(x_m)\circ\tilde{\theta}_{x_1,\dots,x_{m-1}} $ are independent  under~$\Pbarre_p$.
\item[d.]  Suppose $t\le m$, $A\in\mathcal{F}_t$, $B\in\mathcal{F}$
$$\Pbarre_p(A, \tilde{\theta}_{x_1,\dots,x_m}^{-1}(B))=\Pbarre_p(A)\Pbarre_p(B).$$
\item[e.] For every $x\in\Zd$, $\displaystyle \mu_p(x)=\miniop{}{\lim}{n\to +\infty}\frac{\Ebarre_p(\sigma(nx))}n =\inf_{n \ge 1} \frac{\Ebarre_p(\sigma(nx))}n$.
\item[f.] There exists $\alpha,\beta>0$ such that 
$$\forall x \in \Zd \quad \Ebarre_p(\exp(\alpha \sigma(x))\le \exp(\beta(\|x\|_1\vee 1)).$$
\end{enumerate}
\end{prop}

\begin{proof}
To prove a.-d., it is sufficient to mimic the proofs of Lemma 8 and Corollary~9 in \cite{GM-contact}. 
The convergence~e. has been proved for the contact process in \cite{GM-contact}, Theorem~22. The existence of exponential moments for $\sigma$ has been proved for the contact process in \cite{GM-contact-gd}, Theorem~2. 
\end{proof}

\section{Directional limits along subsequences of regenerating times}

The essential hitting times have good regenerating properties, but by construction (see Proposition \ref{PROP-sigma} e.), the vector $(x, \sigma(x))$ lies close to the border of the percolation cone $\{(y, \mu_p(y)):\; y \in \Rd\}$. We now need to build new regenerating points such that the set of directions of these points is dense inside the percolation cone.

We define, for $(y,h) \in \Zd \times \N^*$,  a new regenerating time $s(y,h)$  by setting 
$$s(y,h)=\sigma(y)+ \sum_{i=1}^h \sigma(0) \circ \tilde \theta^{i-1}(0) \circ \tilde \theta(y),$$
and the associated translation: 
\begin{equation*}
\hat \theta_{(y,h)}(\omega) = 
\begin{cases} \theta_{(y,s(y,h))} \omega & \text{if $s(y,h)<+\infty$,}
\\
\omega &\text{otherwise.}
\end{cases}
\end{equation*}
Note that on $\{\tau=+\infty\}$, $(0,0) \to (y,s(y,h)) \to +\infty$ and  $\hat \theta_{(y,h)}=\tilde{\theta}_{y,0,\dots,0}$ (with $h$ zeros). We can easily deduce from Proposition \ref{PROP-sigma} the following properties of the  time $s(y,h)$ under $\Pbarre_p$:

\begin{lemm}
\label{LEMM-sigma}
Fix $p>\pcdir(d+1)$, and  $(y,h) \in \Zd \times \N^*$.
\begin{enumerate}
\item[a.] The probability measure $\Pbarre_p$ is invariant under the translation $\hat \theta_{(y,h)}$.
\item[b.]  The random variables $(s(y,h) \circ (\hat \theta_{(y,h)})^j)_{j \ge 0}$ are independent and identically distributed under~$\Pbarre_p$.
\item[c.]  The measure-preserving dynamical system $(\Omega,\mathcal{F},\Pbarre_p,\hat \theta_{(y,h)})$ is mixing.
\item[d.]  There exists $\alpha,\beta>0$ such that 
$$\forall y \in \Zd \quad \forall h \in \N^* \quad \Ebarre_p(\exp(\alpha s(y,h)))\le \exp(\beta((\|y\|_1\vee 1)+h)).$$
\end{enumerate}
\end{lemm}

We fix $(y,h) \in \Zd\times \N^*$. We work under $\Pbarre_p$, and we set, for every $n \ge 1$,
\begin{eqnarray}
S_n & = & S_n(y,h)=\sum_{k=0}^{n-1} s(y,h)\circ \hat \theta_{(y,h)}^k. \label{DEFI-Sn} 
\end{eqnarray}
The points $(ny,S_n(y,h))_{n \ge 1}$ are the sequence of regenerating points associated to  $(y,h)$ along which we are going to look for subadditivity properties. As, under $\Pbarre_p$, the random variables $(s(y,h) \circ \hat \theta_{(y,h)}^j)_{j \ge 0}$ are independent and identically distributed  with finite first moment (see Lemma \ref{LEMM-sigma}), the strong law of large numbers ensures that $\Pbarre_p$-almost surely
\begin{equation}
\label{EQ-LGNsigma}
\lim_{n \to +\infty} \frac{S_n(y,h)}{n} =\Ebarre_p(s(y,h))=\Ebarre_p(\sigma(y))+h \Ebarre_p(\sigma(0)).
\end{equation}
Thus, for large $n$, the point $(ny,S_n(y,h))$ is not far from the line $\R(y, \Ebarre_p(s(y,h)))$. 

To obtain directional limits along subsequences, we first apply Kingman's subadditive ergodic theorem to $f_n=-\log N_{(ny, S_n(y,h))}$ for a fixed $(y,h) \in \Zd \times \N^*$.
\begin{lemm} 
\label{LEMM-SsAdd}
Fix $p>\pcdir(d+1)$ and $(y,h) \in \Zd\times \N^*$. 
There exists $\alpha_p(y,h) \in(0, \log(2d+1)]$ such that $\Pbarre_p$-almost surely and in $L^1(\Pbarre_p)$, 
$$\lim_{n \to +\infty}\frac{1}{S_n(y,h)} \log N_{(ny,S_n(y,h))} 
= \alpha_p(y,h).$$
\end{lemm}

\begin{proof} 
Fix $(y,h) \in \Zd\times \N^*$. To avoid heavy notations, we omit all the dependence in $(y,h)$. For instance $S_n=S_n(y,h)$ and $\hat{\theta} =\hat{\theta}_{(y,h)}$.
Note that by definition, $\Pbarre_p$-almost surely, for every $n \ge 1$, $(0,0) \to (ny,S_n)\to +\infty$ and consequently,
$N_{(ny,S_n)} \ge 1$. 
For
$n \ge 1$, we set 
$$f_n= -\log N_{(ny,S_n)}.$$
Let $n,p \ge 1$. Note that $S_n+S_p \circ \hat{\theta}_{(y,h)}^n=S_{n+p}$.
As $N_{(py,S_p)}\circ \hat{\theta}^n$ counts the number of open paths from $(ny,S_n)$ to $((n+p)y, S_n+{S_p}\circ \hat{\theta}^n)$, concatenation of paths ensures that 
$N_{(ny,S_n)} \times N_{(py,S_p)}\circ \hat{\theta}^n \le N_{((n+p)y,S_{n+p})}$
which implies that 
$$\forall n,p \ge 1 \quad f_{n+p}\le f_n+f_p \circ \hat{\theta}^n.$$
As $1 \le N_{(ny,S_n)} \le (2d+1)^{S_n}$, 
$$-S_n\log(2d+1)\le f_n\le 0.$$
The integrability of $s$ thus implies the integrability of every $f_n$. So we can apply Kingman's subadditive ergodic theorem.
By property c. in Lemma~\ref{LEMM-sigma}, the dynamical system $(\Omega,\mathcal{F},\Pbarre,\hat{\theta})$ is mixing. Particularly, it is ergodic, so the limit is deterministic:   if we define
$$-\alpha'_p(y,h)=\miniop{}{\inf}{n\ge 1}\frac{\Ebarre_p(f_n)}{n},$$ we have  $\Pbarre_p$-almost surely and in $L^1(\Pbarre_p)$: $\displaystyle\lim_{ n \to+\infty} \frac{f_n}{n} =-\alpha'_p(y,h)$.\\
The limit of the lemma follows then directly from \eqref{EQ-LGNsigma} by setting
$$\alpha_p(y,h)=\frac{\alpha'_p(y,h)}{\Ebarre_ps(y,h)}.$$
Finally $\alpha'_p(y,h)\ge \Ebarre_p(-f_1)=\Ebarre_p (\log {N}_{(y,S_1)})$.
Since ${N}_{(y,S_1)}\ge 1$ $\Pbarre_p$-a.s. and ${N}_{(y,S_1)}\ge 2$ with positive probability, it follows that $\alpha'_p(y,h)>0$, and consequently $\alpha_p(y,h)>0$. \\
As $N_{(ny,S_n)} \le (2d+1)^{S_n}$, we see that $\alpha_p(y,h) \le \log (2d+1)$ and that the convergence also holds in $L^1(\Pbarre_p)$.
\end{proof}

We can now introduce a natural candidate for the limit in Theorem \ref{THEO-lebotheoun}:
\begin{equation}
\label{DEFI-alphap}
\alpha_p=\sup \left\{ \alpha_p(y,h): \; (y,h) \in \Zd\times \N^* \right\}<+\infty.
\end{equation}
Indeed, at the logarithmic scale we are working with, we can expect that the dominant contribution to the number $N_n$ of open paths to level $n$ will be due to the number $N_{nz,n}$ of open paths to level $n$ in the direction $(z,1)$ that optimizes the previous limit. Note however that  in our construction, $(y,h)$ has no real geometrical signification, but it is just a useful encoding: as said before, the asymptotic direction of the regenerating point $(ny,S_n(y,h))$ in  $\Zd\times \N$  is  $$\displaystyle \left( \frac{y}{\Ebarre_p(s(y,h)}, 1 \right).$$
To skip from the subsequences to the full limit, we approximate $B_{\mu_p}(0,1)$ with a denumerable set of points: let
\begin{equation}
\label{E:defiDp}
D_p=\left\{ \frac{y}{\Ebarre_p(s(y,h))}: \; y \in \Zd, h \in \N^*\right\}.
\end{equation}

\begin{lemm}
\label{LEMM-dirdense}
For every $p>\pcdir(d+1)$, $\displaystyle B_{\mu_p}(0,1)\subset\overline{ D_p}$ .
\end{lemm}

\begin{proof}
Note that the set $\{z/l:\; (z,l)\in \Zd \times \N^* \text{ and }\mu_p(z) <l\}$ is dense in $B_{\mu_p}(0,1)$.
Thus fix $(z,l) \in \Zd \times \N^*$ such that $\mu_p(z) <l$ and consider
$$(y_n,h_n)=\left( nz, \left \lceil \frac{n(l-\mu_p(z)}{\Ebarre_p(\sigma(0))}  \right \rceil \right) \in \Zd \times \N^*.$$
Then 
\begin{eqnarray*}
\frac{y_n}{\Ebarre_p(s(y_n,h_n))} & = & \frac{nz}{\Ebarre_p(\sigma(y_n))+h_n\Ebarre_p(\sigma(0))} 
 \to \frac{z}{l}
\end{eqnarray*}
as $n$ goes to $+\infty$.
\end{proof}

Finally,  for $(y,h) \in \Zd\times \N^*$, we denote by
\begin{align}
\label{levarphi}
\forall n \in \N \quad \varphi(n)=\varphi_{(y,h)}(n)=\inf\{k\in \N: \; S_k(y,h) \ge n\}.
\end{align}
Thus, for large $n$, $(\varphi(n).y, S_{\varphi(n)})$ is the first point among the sequence of regenerating points associated to $(y,h)$ to be above level $n$. By the renewal theory, $\Pbarre_p$ almost surely,
\begin{equation}
\label{EQ-LGNphi}
\displaystyle \lim_{n \to +\infty} \frac{ \varphi_{(y,h)}(n)}{n}=\frac{1}{\Ebarre_p(s(y,h))} \text{ and } 
\displaystyle \lim_{n \to +\infty} \frac{ S_{\varphi_{(y,h)}(n)}(y,h)}{n}=1.
\end{equation}
It is also not too far above level $n$:
\begin{lemm}
\label{LEMM-sautemouton}
For every $(y,h) \in \Zd \times \N^*$, there exist positive constants $A,B$ such that 
$$\forall n \in \N \quad \Pbarre(S_{\varphi_{(y,h)}(n)}-n \ge n)\le A\exp(-Bn).$$
\end{lemm}

\begin{proof}
As we work in discrete time, $\varphi(n) \le n$. So 
\begin{eqnarray*}
\Pbarre_p(S_{\varphi(n)}-n \ge n) 
& \le & \Pbarre_p(\exists k \le n: \; s(y,h)\circ \hat \theta_{(y,h)}^k \ge n) \le n \Pbarre_p(s(y,h) \ge n).
\end{eqnarray*}
As $s(y,h)$ admits exponential moments thanks to Lemma \ref{LEMM-sigma}, we can conclude with the Markov inequality.
\end{proof}

\section{Proof of Theorem \ref{THEO-lebotheoun}}

Fix $p>\pcdir(d+1)$. The proof of the almost sure convergence in Theorem \ref{THEO-lebotheoun} is a direct consequence of the forthcoming Lemmas \ref{LEMM-liminf}, \ref{LEMM-limsup} and \ref{LEMM-pareil}. The $L^1$ convergence follows from the remark that $\frac{1}{n} \log  N_n \le \log(2d+1)$.
Remember that $\alpha_p$ is defined in \eqref{DEFI-alphap}.

\begin{lemm} $\Pbarre_p$-almost surely, 
\label{LEMM-liminf}
$\displaystyle \miniop{}{\liminf}{n\to +\infty} \frac{1}{n} \log \overline N_n \ge \alpha_p.$
\end{lemm}

\begin{proof}
Take $(y,h) \in \Zd\times \N^*$.
Note that $(\overline N_n)_{n \ge 1}$ is non-decreasing, and considering the increasing sequence $S_k=S_k(y,h)$, we see that, $\Pbarre_p$ almost surely, for every integer $n$ such that $S_k \le n \le S_{k+1}$, 
$$\frac{1}{n} \log \overline N_n \ge \frac{1}{S_{k+1}} \log \overline N_{S_k} \ge \frac{S_k}{S_{k+1}} \frac{\log \overline N_{(ky,S_k)}}{S_k}.$$
With \eqref{EQ-LGNsigma} and Lemma \ref{LEMM-SsAdd}, we deduce that $\Pbarre_p$ almost surely,
$$\miniop{}{\liminf}{n\to +\infty} \frac{1}{n} \log \overline N_n \ge \alpha_p(y,h),
$$
which completes the proof.
\end{proof}

\begin{lemm} $\Pbarre_p$-almost surely, 
\label{LEMM-limsup}
$\displaystyle \miniop{}{\limsup}{n\to +\infty} \frac{1}{n} \log \overline N_n \le \alpha_p.$
\end{lemm}

\begin{proof} 
Fix $\epsilon>0$ and $\eta \in (0,1)$. We first approximate $B_{\mu_p}(0,1)$ with a finite number of points: with Lemma \ref{LEMM-dirdense}, we can find a finite set $F \subset \Zd \times \N^*$ such that 
\begin{align*}
B_{\mu_p}(0,1+\epsilon) & \subset  \bigcup_{(y,h) \in F} B_{\mu_p} \left( \frac{(1+\epsilon)y}{\Ebarre_p(s(y,h))},(1-\eta)\epsilon/2 \right).
\end{align*}
Then, for $n$ large, we will control the number $\overline{N}_n$ using these directions. We define $M_n(y,h)$ as the first point in the sequence $(ky, S_{(y,h)}(k))_{k\ge 1}$ of regerating points associated to $(y,h)$ to be above level $n(1+\epsilon)$. Using the notation introduced in~\eqref{levarphi}, we set
\begin{eqnarray*}
\forall (y,h) \in F \quad k_n=k_n(y,h) & = &\varphi_{(y,h)}(n(1+\epsilon)), \\
Z_n=Z_n(y,h) & = & k_n.y \in \Zd, \\
V_n=V_n(y,h) &=& S_{k_n}(y,h) \in \N, \\
M_n=M_n(y,h)& = & (Z_n,V_n).
\end{eqnarray*}
For a given $(y,h) \in F$, the law of large numbers \eqref{EQ-LGNphi} says that
\begin{align}
\label{oucest}
k_n(y,h) \sim  \frac{n(1+\epsilon)}{\Ebarre_p(s(y,h))} \text{ and } V_n(y,h) \sim n(1+\epsilon).
\end{align}
So $\Pbarre_p$ almost surely, for all $n$ large enough
$$\forall (y,h) \in F \quad  B_{\mu_p} \left( \frac{(1+\epsilon)ny}{\Ebarre_p(s(y,h))},(1-\eta)\epsilon n/2 \right) \subset B_{\mu_p} \left(  Z_n(y,h) ,  (1-\eta) \epsilon n \right).$$
It follows then from the shape theorem \eqref{E-shapetheo}, that $\Pbarre_p$ almost surely, for all $n$ large enough
\begin{equation}
\label{EQ-unun}\xi_n\subset
B_{\mu_p}(0,(1+\epsilon)n)\subset \miniop{}{\cup}{(y,h) \in F} B_{\mu_p} \left(  Z_n(y,h) ,  (1-\eta) \epsilon n \right).
\end{equation} 
The strategy is to prove that for $n$ large enough, for each $x\in B_{\mu_p}(0,n(1+\epsilon))$, the $n$ first steps of an open path that goes from $(0,0)$ to $(x,n)$ and then to infinity are also the $n$ first steps  of an open path which contributes to $N_{M_n(y,h)}$ for any $(y,h) \in F$ such that $x\in B_{\mu_p}(Z_n(y,h),(1-\eta)\epsilon n)$. To do so, we will use the coupled zone.

Note 
$$G_n=\miniop{}{\cap}{ M\in \{-2n,\dots,2n\}^{d}\times\{0,\dots,2n\}} \left\{\begin{array}{c}\tau<n(1+\epsilon)\\ \text{ or }K'_{n\epsilon}\supset B_{\mu_p}(0,(1-\eta)\epsilon n)\cap\Zd\end{array}\right\} \circ \theta^{\downarrow}_{M}.$$
Since $\theta^{\downarrow}_{M}$ preserves $\P_p$, we easily deduce from~\eqref{EQ-tauexp}, Proposition \ref{PROP-GD} and a Borel--Cantelli argument that $\Pbarre_p$ almost surely, $G_n$ holds for every $n$ large enough.

Now take $n$ large enough such that~\eqref{EQ-unun} holds, $G_n$ holds, together with $V_n(y,h)\le 2n$ for each $(y,h)\in F$, which is possible thanks to~\eqref{oucest}.

Fix $x\in \xi_n$ such that $(x,n)\to \infty$. As  \eqref{EQ-unun} holds,  choose $(y,h)\in F$ such that
$x\in B_{\mu_p}(Z_n(y,h),(1-\eta)\epsilon n)$.
Since $(0,0)\to M_n$ and $V_n\ge n(1+\epsilon)$, we know that $\tau \circ \theta^{\downarrow}_{M_n}\ge n(1+\epsilon)$.
Since $M_n\in\{-2n,\dots,2n\}^{d}\times\{0,\dots,2n\}$, $\mu_p(x-Z_n)\le  (1-\eta)\epsilon n$ and
$G_n$ holds, we have $x-Z_n\in K'_{n\epsilon} \circ \theta^{\downarrow}_{M_n}$. Note that $V_n(y,h) \ge n(1+\epsilon)$, so  $V_n(y,h)-n \ge \epsilon n$. Note also that $(x,n)\to \infty$ implies that
$x-Z_n\in \xi^{\Zd}_{V_n(y,h)-n} \circ \theta^{\downarrow}_{M_n}$.
By definition of the coupled zone, we have
$x-Z_n\in \xi^{0}_{V_n(y,h)-n} \circ \theta^{\downarrow}_{M_n}$.
Going back to the initial orientation, it means that $(x,n)\to M_n$.
So, if $\gamma$ is a path from $(0,0)$ to $(x,n)$, it is clear that
$\gamma$ is the restriction of a path that goes from $(0,0)$ to $M_n$, and then to infinity. Then,
\begin{align*}
\overline{N}_n &  \le \sum_{(y,h) \in F}\overline{N}_{M_n(y,h)}.
\end{align*} 
\begin{figure}[h!]
\begin{center}
\includegraphics[scale=0.7]{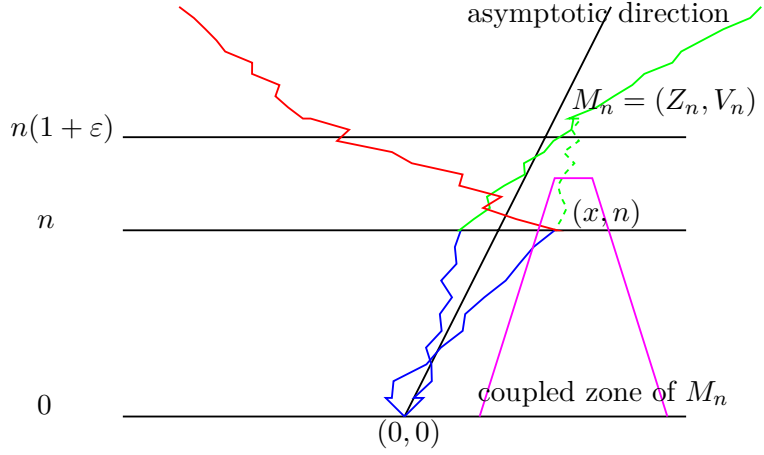}
\put(-165,5){$(0,0)$}
\put(-92,88){$(x,n)$}%
\put(-92,130){$M_n=(Z_n,V_n)$}%
\put(-127,20){coupled zone of $M_n$}%
\put(-131,163){asymptotic direction}%
\put(-302,120){$n(1+\epsilon)$}%
\put(-292,85){$n$}%
\put(-292,15){$0$}%
\caption{Red paths are the part above $n$ of paths from $(0,0)$ to infinity. Green paths are the part above $n$ of paths from $(0,0)$ to infinity  that meet some $M_n$. To bound $\overline{N}_n$, we prove that each start of a red path is the start of a green path.}
\label{unefigure} 
\end{center}
\end{figure}
Next, we use the directional limits given by Lemma \ref{LEMM-SsAdd}: $\Pbarre_p$ almost surely,
$$ \forall (y,h) \in F \quad \lim_{n \to +\infty} \frac{1}{V_n(y,h)} \log \overline{N}_{M_n(y,h)} = \alpha_p(y,h).$$
As  $V_n(y,h) \sim n(1+\epsilon)$, we obtain from the shape theorem~\eqref{E-shapetheo} that $\Pbarre_p$ almost surely, for all $n$ large enough
\begin{equation*}
\forall (y,h) \in F \quad \frac{1}{n (1+\epsilon)} \log \overline{N}_{M_n(y,h)} \le \alpha_p(y,h) + \epsilon \le \alpha_p+\epsilon.
\end{equation*} 
So, for $n$ large enough, we have $\Pbarre_p$ almost surely,
\begin{align*}
\overline{N}_n &\le \sum _{(y,h) \in F}\overline{N}_{M_n(y,h)}
 \le |F|\exp( (\alpha_p+\epsilon)n(1+\epsilon)), \\
\text{so } \miniop{}{\limsup}{n \to +\infty} \frac{1}{n} \log(\overline N_n) & \le  (1+\epsilon)(\alpha_p+\epsilon).
\end{align*} 
We complete the proof by letting $\epsilon$ go to $0$.
\end{proof}

Finally, we prove that working with open paths or with open paths that are the beginning of an infinite open path is essentially the same:
\begin{lemm}
\label{LEMM-pareil} $\Pbarre_p$-almost surely, 
$$\miniop{}{\limsup}{n\to +\infty} \frac{\log N_n}{n} = \miniop{}{\limsup}{n\to +\infty} \frac{\log \overline{N}_n}{n} 
\quad \text{ and } \quad \miniop{}{\liminf}{n\to +\infty} \frac{\log N_n}{n} = \miniop{}{\liminf}{n\to +\infty} \frac{\log \overline{N}_n}{n} .$$
\end{lemm}

\begin{proof}
Fix $ 0<\epsilon<1$ and define, for $n \ge 1$,  the following event
$$E_n=\miniop{}{\cap}{\|z\|_1\le n} \{\tau<\epsilon n\text{ or }\tau=+\infty\}\circ \theta_{(z,\lfloor n(1-\epsilon)\rfloor)}.$$
Assume that $E_n$ occurs. Consider a path $\gamma= (\gamma_i,i)_{0 \le i \le n}$ from $(0,0)$ to $\mathbb Z^d \times \{n\}$ and set $z=\gamma_{\lfloor n(1-\epsilon)\rfloor}$: as $\tau\circ \theta_{(z,\lfloor n(1-\epsilon)\rfloor)} \ge \epsilon n$, the event $E_n$ implies that $\tau\circ \theta_{(z,\lfloor n(1-\epsilon)\rfloor)} =+\infty$. So $(\gamma_i,i)_{0 \le i \le \lfloor n(1-\epsilon)\rfloor}$ contributes to $\overline{N}_{\lfloor n(1-\epsilon)\rfloor}$ and thus, on $E_n$,
\begin{align*}
& N_n\le (2d+1)^{\epsilon n+1} \overline{N}_{\lfloor n(1-\epsilon)\rfloor}, \\
\text{so } \; \frac{1}{n} \log N_n & \le \left( \epsilon+\frac1{n} \right) \log  (2d+1)+ \frac{1}{n} \log \overline{N}_{\lfloor n(1-\epsilon)\rfloor} \\
& \le \left( \epsilon+\frac1{n} \right) \log  (2d+1)+ \frac{1}{\lfloor n(1-\epsilon)\rfloor} \log \overline{N}_{\lfloor n(1-\epsilon)\rfloor}.
\end{align*}
The exponential estimate \eqref{EQ-tauexp} ensures  that
$$\forall n \ge 1 \quad \P_p(E_n^c) \le C_d A n^d \exp(-B\epsilon n)\le A'\exp(-B' n).$$
With the Borel--Cantelli lemma, this leads to:
$$\miniop{}{\limsup}{n\to +\infty} \frac{1}{n} \log N_n \le \epsilon \log  (2d+1)+\miniop{}{\limsup}{n\to +\infty} \frac{1}{n} \log \overline{N}_n.$$
By taking $\epsilon$ to $0$, we obtain
$$\miniop{}{\limsup}{n\to +\infty} \frac{\log {N}_n}{n} \le\miniop{}{\limsup}{n\to +\infty} \frac{\log \overline{N}_n}{n}.$$
 The proof for the inequality with $\liminf$ instead of $\limsup$ is identical.
Since $\overline{N}_n \le N_n$, the reversed inequalities are obvious. 
\end{proof}


\section{Proof of Theorem \ref{THEO-nombreAn}}

\subsection{Construction and continuity of $\tilde{\alpha}_p$}
Recall that $D_p$ was defined in~\eqref{E:defiDp}.
Our strategy is to prove that the identity
$$\tilde \alpha_p \left( \frac{y}{\Ebarre_p(s(y,h))}\right)= \alpha_p (y,h).$$
defines a map on $D_p$ that is uniformly continuous on every compact subset of $D_p\cap \mathring{B}_{\mu}(0,1)$. We first refine the argument of Lemma \ref{LEMM-limsup} implying the coupled zone:

\begin{lemm}
\label{LEMM-continuitemiam}
Let $\beta \in (0,1)$. There exists $\alpha>0$ such that the following holds.
For every $\epsilon>0$, for every $\hat x_1, \hat x_2 \in B_{\mu_p}(0, 1-\beta)$,
 if
$$\mu_p \left( \hat x_1 - \hat x_2 \right) \le \alpha \epsilon,$$
then  for any sequences of points $(M_n^1=(Z_n^1,V_n^1))_n$ and $(M_n^2=(Z_n^2,V_n^2))_n$ in $\Zd \times \N^*$, for any $C>0$ such that, $\Pbarre_p$ almost surely,  
\begin{eqnarray*}
\frac{Z_n^1}{V_n^1} \to \hat x_1 & \text{ and } & \frac{V_n^1}{n} \to C(1+\epsilon), \\
\frac{Z_n^2}{V_n^2} \to \hat x_2 & \text{ and } & \frac{V_n^2}{n} \to C,
\end{eqnarray*}
we have the following property: $\Pbarre_p$ almost surely,  for every $n$ large enough, if $(0,0) \to (Z_n^1,V_n^1)$ and $(0,0) \to (Z_n^2,V_n^2) \to \infty$, then $\overline N_{(Z_n^2,V_n^2)} \le N_{(Z_n^1,V_n^1)}$.
\end{lemm}

\begin{proof}
Fix  small $\alpha, \eta>0$ and a large integer $K \ge 3$ such that
$$ \alpha + (1-\beta) <\frac{K-2}{K}(1- \eta). $$
Fix $\epsilon>0$. Set $\epsilon'=\epsilon/K$.

\begin{eqnarray*}
\mu_p \left( \frac{Z_n^2}{Cn}- \frac{Z_n^1}{Cn}\right) 
& \le & \mu_p \left( \frac{Z_n^2}{V_n^2}\right)\left| \frac{V_n^2}{Cn}- 1\right|
+ \mu_p \left( \frac{Z_n^2}{V_n^2}- \hat x_2 \right) 
 + \mu_p \left(\hat x_2 - \hat x_1 \right)\\
& & + \mu_p \left(\hat x_1-\frac{Z_n^1}{V_n^1}  \right)  +  \mu_p \left( \frac{Z_n^1}{V_n^1} \right)\left| \frac{V_n^1}{Cn}- 1\right|.
\end{eqnarray*}
So $\Pbarre$ almost surely, 
$$ \miniop{}{\limsup}{n \to +\infty} \mu_p \left( \frac{Z_n^2}{Cn}- \frac{Z_n^1}{Cn}\right) 
 \le  \left( \alpha + 1 - \beta \right)  \epsilon <\frac{K-2}{K}(1- \eta)\epsilon, 
$$
so $\Pbarre$ almost surely, for every $n$ large enough, 
\begin{equation}
\label{EQ-seq1}
 \mu_p \left( \frac{Z_n^2}{Cn}- \frac{Z_n^1}{Cn}\right) 
 \le \frac{K-2}{K}(1- \eta)\epsilon=(K-2)(1- \eta)\epsilon'.
\end{equation}
By the convergences for the $V_n^i/n$, we know that $\Pbarre_p$ almost surely, for every $n$ large enough, 
\begin{equation}
\label{EQ-seq2}
|V_n^1-Cn(1+\epsilon)| \le Cn\epsilon' \text{ and } |V_n^2-Cn| \le Cn\epsilon'.
\end{equation}
Define
\begin{eqnarray*}
G_n & = &   
 \{ \forall x \in [-Cn(1+2\epsilon), Cn(1+2\epsilon)]^d \quad \forall k \in [Cn(1+\epsilon-\epsilon'),Cn(1+\epsilon+\epsilon')]  \\
&& \quad (\tau\circ\theta^{\downarrow}_{(x,k)}\ge \epsilon' Cn) \Rightarrow \forall m \ge \epsilon' Cn \;
B_{\mu_p}(0, x,(1-\eta)m )   \subset \tilde K'_{m} \circ\theta^{\downarrow}_{(x,k)}\}.
\end{eqnarray*}
With the large deviations for the coupled zone given in Proposition \ref{PROP-GD}, there exist $A,B>0$ such that
$$\forall n\text{ large enough} \quad  \Pbarre_p(G_n^c) \le A \exp(-Bn).$$
Thus, the Borel--Cantelli lemma ensures that $\Pbarre_p(\liminf \;G_n)=1.$

Assume then that $\tau^0=+\infty$. $\Pbarre_p$ almost surely, for every $n$ large enough, we know that \eqref{EQ-seq1}, \eqref{EQ-seq2} and $G_n$ occur.  
 Assume that, for one of these large enough $n$,  $(0,0) \to (Z_n^1,V_n^1)$ and $(0,0) \to (Z_n^2,V_n^2) \to \infty$. Note that 
$$V_n^1-V_n^2 \ge Cn(1+\epsilon-\epsilon')-Cn(1+\epsilon') \ge Cn(K-2) \epsilon'.$$
So, on the event $G_n$, as $(0,0) \to (Z_n^1,V_n^1)$, we see that $\tau\circ \theta^{\downarrow}_{M_n^1} \ge \epsilon' Cn$, so
$$K_{V_n^1-V_n^2}\circ \theta^{\downarrow}_{M_n^1}  \supset B_{\mu_p}(Z_n^1, (1-\eta)C(K-2)n \epsilon').$$
So, with \eqref{EQ-seq1}, we see that ${Z_n^2} \in K_{V_n^1-V_n^2}\circ \theta^{\downarrow}_{M_n^1} $. As $(Z_n^2,V_n^2) \to \infty$, then $(Z_n^2,V_n^2) \to (Z_n^1,V_n^1)$, which gives an injection from the set of open paths from $(0,0)$ to $(Z_n^2,V_n^2)$ into the set of open paths from $(0,0)$ to $(Z_n^1,V_n^1)$.
\end{proof}

For $(y,h) \in \Zd \times \N^*$, we define $M_n(y,h)$ as the first point in the sequence $(ky, S_{(y,h)}(k))$ of regerating points associated to $(y,h)$ to be above level $n$ (see Definition \eqref{levarphi}): \
\begin{eqnarray*}
k_n=k_n(y,h) & = &\varphi_{(y,h)}(n), \\
Z_n=Z_n(y,h) & = & k_n.y \in \Zd \quad \text{ and } \quad V_n=V_n(y,h)=S_{k_n}(y,h) \in \N, \\
M_n=M_n(y,h)& = & (Z_n,V_n).
\end{eqnarray*}
The law of large numbers \eqref{EQ-LGNphi} says that $\Pbarre_p$ almost surely,
$$\frac{Z_n(y,h)}{n}=\frac{k_n(y,h).y}{n}
 \sim  \frac{y}{\Ebarre_p(s(y,h))} \quad \text{ and } \quad \frac{V_n}{n}\sim 1.$$
 The next lemma is a first step toward continuity.
\begin{lemm}
\label{LEMM-continuity} 
Let $\beta \in (0,1)$. There exists $\alpha>0$ such that, for every $\epsilon>0$, for every $(y_1,h_1), (y_2,h_2) \in \Zd \times \N^*$ such that
$$
\mu_p \left( \frac{y_1}{\Ebarre_p(s(y_1,h_1))}\right) \le 1- \beta \text{ and } \mu_p \left(  \frac{y_2}{\Ebarre_p(s(y_2,h_2))} \right) \le 1- \beta, 
$$
 if
$\displaystyle \mu_p \left( \frac{y_1}{\Ebarre_p(s(y_1,h_1))} - \frac{y_2}{\Ebarre_p(s(y_2,h_2))} \right) \le \alpha \epsilon$, \\then 
 $| \alpha_p(y_1,h_1) - \alpha_p(y_2,h_2) | \le \epsilon$.
\end{lemm}

\begin{proof} Take, for every $n$, $(Z_n^1,V_n^1)=M_{n(1+\epsilon)}(y_1,h_1)$ and $(Z_n^2,V_n^2)=M_{n}(y_2,h_2)$. With the previous Lemma, we obtain
\begin{eqnarray*}
\overline N_{M_{n}(y_2,h_2)} & \le  &\overline N_{M_{n(1+\epsilon)}(y_1,h_1)}, \\
\frac{1}{V_n(y_2,h_2)} \log \overline N_{M_n(y_2,h_2)} & \le  & \frac{H_{n(1+\epsilon)}(y_1,h_1)}{V_n(y_2,h_2)} \frac{1}{H_{n(1+\epsilon)}(y_1,h_1)} \log \overline N_{M_{n(1+\epsilon)}(y_1,h_1)}, \\
\alpha_p(y_2,h_2) & \le & (1+\epsilon)\alpha_p(y_1,h_1).
\end{eqnarray*}
Exchanging the roles of the $(y_i,h_i)$, we obtain
$$|\alpha_p(y_2,h_2)-\alpha_p(y_1,h_1)| \le \epsilon\log (2d+1).$$
\end{proof}

We define the following equivalence relation of the points in $\Zd \times \N^*$: 
$$(y_1,h_1) \sim (y_2,h_2) \quad \Leftrightarrow \quad \frac{y_1}{\Ebarre_p(s(y_1,h_1))} = \frac{y_2}{\Ebarre_p(s(y_2,h_2))}.$$
Lemma \ref{LEMM-continuity} ensures that if
$(y_1,h_1) \sim (y_2,h_2)$, then $ \alpha_p(y_1,h_1)=\alpha_p(y_2,h_2)$.
We can thus define on the quotient set of directions $D_p$, defined in \eqref{E:defiDp},
the following directional limit:
$$\tilde \alpha_p \left( \frac{y}{\Ebarre_p(s(y,h))}\right)= \alpha_p (y,h).$$
Lemma \ref{LEMM-continuity} ensures that the application $\tilde \alpha_p$ is uniformly continuous on each $D_p \cap B_{\mu_p}(0,(1-\beta))$. Note that the $\alpha$ given by Lemma~\ref{LEMM-continuity} gives an upper bound for its modulus of continuity. As, with Lemma \ref{LEMM-dirdense}, $D_p \cap  B_{\mu_p}(0,(1-\beta))$ is dense in the compact set 
$ B_{\mu_p}(0,(1-\beta))$, we can extend $\tilde \alpha_p$ to any $ B_{\mu_p}(0,(1-\beta))$, and then to $\mathring B_{\mu_p}(0,1)$.
 
\subsection{Concavity}

\begin{lemm}
\label{CORO-concave}
The application $\tilde \alpha_p$ is concave on  $\mathring B_{\mu_p}(0,1)$.
\end{lemm}
\begin{proof} Fix $\beta\in (0,1)$, and consider 
$\hat y, \hat y ' \in D_p \cap B_{\mu_p}(0,(1-\beta))$. There exist $(y,h), (y',h') \in \Z^d \times \N$ such that
\begin{align*}
& s = \Ebarre_p(s(y,h)) \quad   & s'  = \Ebarre_p(s(y',h')), \\
& \hat y  = \frac{y}{\Ebarre_p(s(y,h))}=\frac{y}{s} \quad   & \hat y'  = \frac{y'}{\Ebarre_p(s(y',h'))}=\frac{y'}{s'}.
\end{align*}
Let $\alpha$ be given by Lemma \ref{LEMM-continuitemiam}. 
Let $\lambda \in (0,1)$ and  $\epsilon>0$. 
Writing, for $k,k' \ge 1$:
$$
\frac{ky+k'y'}{k s + k' s'} = \frac{k s  }{k s  + k' s'} \hat y +
\frac{ k' s' }{k s  + k' s'}\hat y' \in  B_{\mu_p}(0,1-\beta ),
$$
we can choose $k,k' \ge k_0$ such that :
\begin{align*}
\mu_p \left( \frac{ky+k'y'}{ks+k's'} - 
\left(\lambda \hat y +(1-\lambda')  \hat y' \right) \right) & \le \alpha \epsilon, \\
\text{and } \frac{ks \tilde\alpha_p\left( \hat y \right) + k's' \tilde \alpha_p \left( \hat y' \right)}{ks+k's'} 
& \ge  
\frac1{1+\epsilon}
\left(\lambda \tilde\alpha_p \left( \hat y \right)+(1-\lambda)\tilde\alpha_p\left(\hat y'\right)\right).
\end{align*}
In particular, the uniform continuity of $\tilde \alpha_p$ on $\overline B_{\mu_p}(0,1-\beta )$ ensures that 
\begin{equation}
\label{etdezero}
\tilde \alpha_p \left( \frac{ky+k'y'}{ks+k's'} \right) \le \alpha_p \left(\lambda \hat y +(1-\lambda)  \hat y' \right) +\epsilon.
\end{equation}
By Lemma \ref{LEMM-dirdense}, we can choose $(y'',h'') \in \Z^d\times\N^*$ such that (with the same notation):
$$
\mu_p\left(\hat y''- \frac{ky+k'y'}{ks+ k's'}\right) \le \alpha\epsilon \quad \text{ and } \quad 
\mu_p\left(\hat y''\right) \le 1-\beta.
$$
In particular, the uniform continuity of $\tilde \alpha_p$ on $B_{\mu_p}(0,1-\beta )$ and \eqref{etdezero} ensure that 
\begin{equation}
\label{etdezeroun}
\tilde \alpha_p (\hat y'') \le \tilde \alpha_p \left( \frac{ky+k'y'}{ks+k's'} \right) +\epsilon\le \alpha_p \left(\lambda \hat y +(1-\lambda)  \hat y' \right)+2\epsilon.
\end{equation}
For every $n \ge 1$, set
\begin{align*}
Z^2_n & = nky+nk'y', \\
V_n^2 & = S_{kn}(y,h) + S_{k'n}(y',h') \circ \hat\theta^{kn}_{(y,h)}
\end{align*}
By the law of large numbers, $S_{kn}(y,h)/n$ and $S_{k'n}(y',h')/n$ almost surely converge, respectively, to $ks$ ans $k's'$. Since $s(y',h')\in L^2(\Pbarre_p)$, complete convergence actually occurs (see e.g.~Hsu-Robbins~\cite{MR0019852}), so, since the $\hat\theta$ operators leave $\Pbarre_p$ invariant, we get the following $\Pbarre_p$ almost sure convergences:
\begin{align*}
\frac{V_n^2}n & \to  ks + k's' \quad 
\text{ and } \quad \frac{Z_n^2}{V_n^2} \to \frac{ky+k'y'}{ks + k's'}.
\end{align*}
Note that $(0,0) \to (nky,S_{kn}(y,h)) \to (Z^2_n,V_n^2) \to +\infty$. In particular, 
\begin{align*}
\overline N_{Z^2_n,V_n^2} & \ge \overline N_{nky,S_{kn}(y,h)} \times 
\overline N_{Z^2_n,V_n^2} \circ \hat\theta^{kn}_{(y,h)}, \\
\frac{\log \overline N_{Z^2_n,V_n^2}}{n(ks + k's')}  & \ge 
\frac{nks}{n(ks + k's')} \frac{\log \overline N_{nky,S_{kn}(y,h)}}{nks}  + \left( \frac{nk's'}{n(ks + k's')} \frac{\overline N_{nk'y', S_{k'n}(y',h')}}{nk's' } \right)\circ \hat\theta^{kn}_{(y,h)}.
\end{align*}
Using the $L^1$ convergence in the subadditive lemma \ref{LEMM-SsAdd}, we then obtain
\begin{align}
\liminf \; \Ebarre_p \left( \frac{\log \overline N_{Z^2_n,V_n^2}}{n(ks + k's')}\right) & \ge  \frac{ks}{ks + k's'}\tilde \alpha_p(\hat y)+ \frac{k's'}{ks + k's'} \tilde \alpha_p(\hat y') \nonumber \\
& \ge  
\frac{\lambda \tilde\alpha_p \left( \hat y \right)+(1-\lambda)\tilde\alpha_p\left(\hat y'\right)}{1+\epsilon}
. \label{etdeun}
\end{align}
For every $n \ge 1$, we now set $\displaystyle 
(Z_n^1,V_n^1) = M_{(ks + k's')n(1+\epsilon)}(y'',h'').
$
We have the following $\Pbarre_p$ almost sure convergence:
$$\frac{V_n^1}n \to (1+\epsilon)(ks + k's')
\quad \text{ and } \quad 
\frac{Z_n^1}{V_n^1} \to \frac{y''}{s''}.
$$
By Lemma \ref{LEMM-continuitemiam} we get,
$\Pbarre_p$ almost surely, for large enough $n$:
\begin{equation}
\label{etdedeux}
\overline{N}_{(Z_n^2,V_n^2)} \le \overline{N}_{(Z_n^1,V_n^1)}.
\end{equation}
With the $L^1$ convergence in Lemma \ref{LEMM-SsAdd} we get that :
$$
\lim_{n \to +\infty} \Ebarre_p \left( \frac1{n(ks + k's')} \log\left(\overline{N}_{(Z_n^1,V_n^1)}\right) \right) = (1+\epsilon) \tilde \alpha_p(\hat y'').
$$
With \eqref{etdezeroun}, \eqref{etdeun} and \eqref{etdedeux}, we obtain
$$(1+\epsilon)\left( \alpha_p \left(\lambda \hat y +(1-\lambda)  \hat y' \right)+2\epsilon \right)  \ge (1+\epsilon) \tilde \alpha_p(\hat y'') \ge \frac{\lambda \tilde\alpha_p \left( \hat y \right)+(1-\lambda)\tilde\alpha_p\left(\hat y'\right)}{1+\epsilon}
.$$
Letting $\epsilon$ tend to $0$, we get the concavity of $\tilde\alpha_p$ on $D_p \cap B_{\mu_p}(0,(1-\beta))$. Since $\tilde\alpha_p$ is continuous, a density argument completes the proof. 
\end{proof}

\subsection{Inequalities for the directional convergence} 

We now prove refined versions of Lemmas \ref{LEMM-liminf} and \ref{LEMM-limsup}.

\begin{lemm} For every subset $A$  of  $\mathring B_{\mu_p}(0,1)$ such that $\mathring A\ne \varnothing$, $\Pbarre_p$-almost surely, 
\regine{j'ai rajoute $\mathring A\ne \varnothing$}
\label{LEMM-liminfA}
$$\displaystyle \miniop{}{\liminf}{n\to +\infty} \frac{1}{n} \log \overline N_{nA,n} \ge \sup_{x \in \mathring{A}} \tilde{\alpha}_p(x).$$
\end{lemm}

\begin{proof}
Let $L\in\R$ with $L<\sup_{x \in \mathring{A}} \tilde{\alpha}_p(x)$. There exists
$x\in \mathring{A}$ with $\tilde{\alpha}_p(x)>L$.
Fix $\epsilon \in (0,1)$ such that $B(x,8\epsilon)\subset A$. By the continuity of $\tilde \alpha_p$, if we take $\epsilon$ small enough, we can also ensure that
$\tilde{\alpha}_p>L$ on $B(x,8\epsilon)$.
With Lemma~\ref{LEMM-dirdense}, we can find $(y,h) \in \Zd\times \N$  such that $\displaystyle \hat y =\frac{y}{\Ebarre_p(s(y,h))}\in B(x,4\epsilon)$. 

We define $M_n(y,h)$ as the first point in the sequence $(ky, S_{(y,h)}(k))_{k\ge 1}$ of regerating points associated to $(y,h)$ to be above level $n(1-\epsilon)$. Using the notation introduced in~\eqref{levarphi}, we set
\begin{eqnarray*}
\forall (y,h) \in F \quad k_n=k_n(y,h) & = &\varphi_{(y,h)}(n(1-\epsilon)), \\
Z_n=Z_n(y,h) & = & k_n.y \in \Zd \quad \text{ and } \quad V_n=V_n(y,h)=S_{k_n}(y,h) \in \N, \\
M_n=M_n(y,h)& = & (Z_n,V_n).
\end{eqnarray*}
The law of large numbers \eqref{EQ-LGNphi} says that
\begin{align}
\label{oucestA}
Z_n(y,h) \sim  n(1-\epsilon)\hat y \text{ and } V_n(y,h) \sim n(1-\epsilon).
\end{align}
Note 
$$G_n = \bigcap_{\substack{M\in B_{\mu_p} ( n(1-\epsilon)\hat y, \epsilon n ) \times[n(1-\epsilon)..n(1-\epsilon/2)], \\
k \ge \epsilon n /2}}
 \{ \xi_k^0 \subset B_{\mu_p}(0, (1+\epsilon) k) \} \circ \theta_{M}.$$
Since $\theta_{M}$ preserves $\P_p$, we easily deduce from~\eqref{EQ-tauexp}, Proposition \ref{PROP-GD} and a Borel--Cantelli argument that $\Pbarre_p$ almost surely, $G_n$ holds for $n$ large enough.

Now take $n$ large enough such that $G_n$ holds and, with \eqref{oucestA},
$Z_n \in B_{\mu_p} ( n(1-\epsilon)\hat y, \epsilon n )$ and $(1-\epsilon)n \le V_n \le (1-\epsilon/2)n$, so that $\epsilon n /2 \le n-V_n \le \epsilon n $. Then $G_n$ ensures that ($\epsilon <1$)
$$\xi^{Z_n}_{n-V_n} \subset B_{\mu_p}(Z_n, (1+\epsilon) \epsilon n)  \subset B_{\mu_p}(n(1-\epsilon)\hat y, 3 \epsilon n) \subset B_{\mu_p}(n\hat y, 4 \epsilon n) \subset n\mathring A.$$
So $\overline N_{M_n} \le \overline N_{nA,n}$, and then
$$\frac{1}{n} \log \overline N_{nA,n} \ge \frac{V_n}{n} \frac{1}{V_n} \log \overline N_{M_n}.$$
With \eqref{oucestA} and Lemma \ref{LEMM-SsAdd}, we deduce that $\Pbarre_p$ almost surely,
$$\miniop{}{\liminf}{n\to +\infty} \frac{1}{n} \log \overline N_{nA,n} \ge
 \frac{1}{1-\epsilon}\alpha_p(\hat y) \ge \frac{1}{1-\epsilon}L.
$$
Letting $\epsilon$ going to $0$  completes the proof.
\end{proof}

\begin{lemm} For every non-empty set $A$ such that $\overline A \subset \mathring B_{\mu_p}(0,1)$, $\Pbarre_p$-almost surely,
\label{LEMM-limsupA}
$$\displaystyle \miniop{}{\limsup}{n\to +\infty} \frac{1}{n} \log \overline N_{nA,n} \le \sup_{x \in {A}} \tilde{\alpha}_p(x).$$
\end{lemm}
   
\begin{proof} 
The proof is a refinement of that of Lemma \ref{LEMM-limsup}.
Let $\delta>0$. Since $\overline{A}$ is a compact subset of $\mathring B_{\mu}(0,1)$ and
$z\mapsto \tilde{\alpha}_p(z)$ is continuous on $\mathring B_{\mu}(0,1)$, one can find 
$\epsilon \in (0,1)$ such that $$\miniop{}{\sup}{A+B_{\mu_p}(0,2\epsilon)} \tilde{\alpha}_p\le \delta+\miniop{}{\sup}{ A} \tilde{\alpha}_p.$$
Now take $\eta>0$ and~$F$ as defined in the proof of Lemma \ref{LEMM-limsup} and note 
$$F_{B}=\left\{(y,h)\in F:\;  \frac{y}{\Ebarre_p(s(y,h))}\in B\right\}.$$
Now consider $x\in nA$. Since $nA\subset B_{\mu_p}(0,n(1+\epsilon))$, for $n$ large enough, we can find $(y,h)\in F$ such that
$x/n\in B_{\mu_p} \left( \frac{(1+\epsilon)y}{\Ebarre_p(s(y,h))},(1-\eta)\epsilon/2 \right)$. We have 
\begin{align*}
\mu_p\left(\frac{y}{\Ebarre_p(s(y,h))}-\frac{x}n\right)&\le \mu_p\left((1+\epsilon)\frac{y}{\Ebarre_p(s(y,h))}-(1+\epsilon)\frac{x}n\right)\\
& \le \mu_p\left(\frac{(1+\epsilon)y}{\Ebarre_p(s(y,h))}-\frac{x}n\right)+\epsilon\mu_p(x/n)\\
&\le (1-\eta)\epsilon/2+\epsilon\mu_p(x/n) \le 2 \epsilon.
\end{align*}
Since $x/n\in A$, we get $(y,h)\in F_{A+B_{\mu_p}(0,2\epsilon)}$.
Now, following the proof of Lemma \ref{LEMM-limsup},  
 for $n$ large enough, for each $x\in nA$, the $n$ first steps of an open path that goes from $(0,0)$ to $(x,n)$ and then to infinity are also the $n$ first steps  of an open path which contributes to $N_{M_n(y,h)}$ for any $(y,h) \in F_{A+B_{\mu_p}(0,2\epsilon)}$, which gives
\begin{align}
\label{larafinee}
\overline{N}_{nA,n}&\le \miniop{}{\sum}{(y,h)\in F_{A+B_{\mu_p}(0,2\epsilon)}} \overline{N}_{M_n(y,h)}.
\end{align}
As previously, we get
$$\miniop{}{\limsup}{n \to +\infty} \frac{1}{n} \log(\overline N_{nA,n})\le\sup_{ F_{A+B_{\mu_p}(0,2\epsilon)}} \alpha_p\le \sup_{ A+B_{\mu_p}(0,2\epsilon)}\tilde{\alpha}_p\le \delta+\miniop{}{\sup}{A} \tilde{\alpha}_p.$$
We complete the proof by letting $\delta$ go to $0$.
\end{proof}

\subsection{Proof of Theorem \ref{THEO-PGD}}
It remains to skip from $\overline N_{nA,n}$ to $N_{nA,n}$.

Fix $ 0<\epsilon<1$ and define, for $n \ge 1$,  the following event
\begin{align*}
G_n & =\miniop{}{\cap}{\|z\|_1\le n} \{\tau<\epsilon n\text{ or }\tau=+\infty\}\circ \theta_{(z,\lfloor n(1-\epsilon)\rfloor)} \\
& \cap \miniop{}{\cap}{\|z\|_1\le n} \{ K'_{\epsilon n} \subset B_{\mu_p}(0,2\epsilon n)\} \circ \theta^{\downarrow}_{(z,n)}.
\end{align*}
As before, a Borel-Cantelli argument ensures that $\Pbarre_p$-almost surely, $G_n$ occurs for every large enough $n$. 

Assume that $G_n$ occurs. Consider a path $\gamma= (\gamma_i,i)_{0 \le i \le n}$ from $(0,0)$ to $nA \times \{n\}$ and set $z=\gamma_{\lfloor n(1-\epsilon)\rfloor}$: as $\tau\circ \theta_{(z,\lfloor n(1-\epsilon)\rfloor)} \ge \epsilon n$, the event $G_n$ implies that $\tau\circ \theta_{(z,\lfloor n(1-\epsilon)\rfloor)} =+\infty$. Looking backwards in time, we see that all these $z$ are in $nA+B_{\mu_p}(0,2\epsilon n)$. So $(\gamma_i,i)_{0 \le i \le \lfloor n(1-\epsilon)\rfloor}$ contributes to $\overline{N}_{nA+B_{\mu_p}(0,2\epsilon n), \lfloor(1-\epsilon) n\rfloor}$ and thus, on $G_n$,
\begin{align*}
 N_{nA,n} & \le (2d+1)^{\epsilon n+1} \overline{N}_{nA+B_{\mu_p}(0,2\epsilon n), \lfloor(1-\epsilon) n \rfloor}, \\
\text{so } \; \frac{1}{n} \log N_{nA,n} & \le \left( \epsilon+\frac1{n} \right) \log  (2d+1)+ \frac{1}{ n} \log \overline{N}_{nA+B_{\mu_p}(0,2\epsilon n), \lfloor(1-\epsilon) n \rfloor}.
\end{align*}
Now, we first use Lemma \ref{LEMM-limsupA} and take the $\limsup$, and then we use the continuity of $\tilde \alpha_p$ and let $\epsilon$ go to $0$:
\begin{align*}
\miniop{}{\limsup}{n\to +\infty} \frac{\log {N}_{nA,n}}{n} & \le \epsilon \log  (2d+1) +\sup_{x \in A+B_{\mu_p}(0,2\epsilon )} \tilde \alpha_p(x), \\
\text{so } \; \miniop{}{\limsup}{n\to +\infty} \frac{\log {N}_{nA,n}}{n} & \le \sup_{x \in A} \tilde \alpha_p(x).
\end{align*}
As $\overline {N}_{nA,n} \le {N}_{nA,n}$, with Lemma \ref{LEMM-liminfA} we obtain that
$$\miniop{}{\lim}{n\to +\infty} \frac{\log {N}_{nA,n}}{n}= \sup_{x \in A} \tilde \alpha_p(x).$$
This completes the proof.

\medskip

\emph{Olivier Garet and R\'egine Marchand would like to warmly thank Matthias Birkner and Rongfeng Sun for pointing out an error in a previous version of the paper.} 


\def\refname{References}
\bibliographystyle{plain}

\end{document}